\documentclass[11pt,oneside]{amsart}
\usepackage{amsmath,ifthen, amsfonts, 
srcltx,
   amsopn, textcomp, txfonts, mathrsfs,}

\usepackage{hyperref}
\usepackage{graphicx}
\usepackage{xypic}
\usepackage{overpic}

\usepackage[small]{caption}
\usepackage{marginnote}
\usepackage{float}
\usepackage{tikz-cd}
\usepackage{scrextend}

\newtheorem{thm}{Theorem}[section]
\newtheorem{lem}[thm]{Lemma}

\newtheorem{cor}[thm]{Corollary}

\newtheorem{prop}[thm]{Proposition}

\theoremstyle{definition}
\newtheorem{defn}[thm]{Definition}
\newtheorem{rem}[thm]{Remark}
\newtheorem{exmp}[thm]{Example}

\newtheorem{prob}[thm]{Problem}

\DeclareMathOperator{\Aut}{Aut}

\DeclareMathOperator{\Sym}{Sym}

\DeclareMathOperator{\cstar}{Star}
\DeclareMathOperator{\cSt}{St}

\DeclareMathOperator{\frakK}{\mathfrak{K}}
\DeclareMathOperator{\fraks}{\mathfrak{s}}
\DeclareMathOperator{\frakf}{\mathfrak{f}}
\DeclareMathOperator{\ad}{ad}

\newcommand{\wt}[1]{\widetilde{#1}}

\newcommand{\id}{\textrm{id}}

 \label{mainB}

\DeclareMathOperator{\lk}{lk}

\DeclareMathOperator{\Core}{Core}

\usepackage{tikz-cd} 
\tikzset{
	labl/.style={anchor=south, rotate=90, inner sep=.5mm}
}

\newcommand\restr[2]{{
  \left.\kern-\nulldelimiterspace 
  #1 
  \vphantom{\big|} 
  \right|_{#2} 
  }}

\begin{document}

\author{ Daniel J.~Woodhouse}
\title{Leighton's Theorem and Regular Cube Complexes}

\address{}
\email{daniel.woodhouse@mail.mcgill.ca}
\subjclass[]{}

\thanks{}

\begin{abstract}
Leighton's graph covering theorem states that two finite graphs with common universal cover have a common finite cover.
We generalize this to a large family of non-positively curved special cube complexes that form a natural generalization of regular graphs.
This family includes both hyperbolic and non-hyperbolic CAT(0) cube complexes.
\end{abstract}

\maketitle

Leighton's graph covering theorem states that two finite graphs with isomorphic universal covers have isomorphic finite covers.
First conjectured by Angluin~\cite{Angluin80} and proven by Leighton~\cite{Leighton82}, whose background was in computer science and the study of networks, the topic has been picked up by topologists and group theorists interested in producing generalizations to graphs with extra structure, including colourings and line patterns (See~\cite{BassKulkarni90,Neumann10,Woodhouse21,Shepherd19}).
Although it is desirable to generalize such a theorem to higher dimensions, counter-examples are known even when the universal cover is the product of two trees.
Standard arithmetric constructions were known to give irreducible lattices acting on the product of trees, and in the 90's non-residually finite and even simple examples were given (see~\cite{BurgerMozes97, WiseThesis}).

A particularly exciting conjecture was made by Haglund in~\cite{Haglund06} that Leighton's graph covering theorem should generalize to special cube complexes.
In the same paper Haglund proved the conjecture for the class of right angled Fuchsian buildings (commonly referred to as ``Bourdon buildings'') and more generally for ``type-preserving'' lattices in the automorphism group of a building associated to a finite graph product of finite groups. 

In this paper we will prove Haglund's conjecture for a large family of CAT(0) cube complexes which exhibit symmetry and homogeneity reminiscent of finite regular trees.
Let $L$ be a finite simplicial flag complex.
An \emph{$L$-cube-complex} $X$ is a cube complex such that every link is isomorphic to $L$.
Given a flag complex, the Davis complex $D(L)$ of the associated right angles Coxeter group is a CAT(0) $L$-cube-complex.
In general, $D(L)$ is not the unique CAT(0) $L$-cube-complex, but in~\cite{Lazarovich18} Lazarovich shows that $D(L)$ is unique if and only if  $L$ is \emph{superstar-transitive}.
A flag complex $L$ is superstar-transitive if for any two simplicies $\sigma , \sigma' \subseteq L$ any isomorphism $\cSt(\sigma) \to \cSt(\sigma')$ sending $\sigma$ to $\sigma'$ extends to an automorphism of $L$.
Lazarovich also showed that in this case $\Aut(X)$ is a virtually simple.

The principal set of examples of superstar transitive flag complexes presented by Lazarovich are \emph{Kneser complexes}.
Let $\Delta$ be a finite set.
The Kneser complex $\frakK_n(\Delta)$ is the simplicial flag complex defined with vertex set the $n$-element subsets of $\Delta$, and edges joining disjoint $n$-element subsets.
In the particular case that $|\Delta| = nd +1$ the Kneser complex is superstar transitive and its automorphism group is precisely the natural action of the permutation group $\Sym(\Delta)$. (See Section~\ref{sec:Kneser}).
We prove the following:

\begin{thm} \label{thm:main}
 Let $n \geq 2, d\geq 1$ and let $\Delta$ be a finite set of cardinality $nd+1$.
 Let $L$ be the Kneser complex $\frakK_n(\Delta)$.
 Suppose that $X_1, X_2$ are compact, $L$-cube-complexes such that all finite index subgroups of the hyperplane subgroups are separable in $\pi_1 X_1$ and $\pi_1 X_2$ respectively.
 Then $X_1$ and $X_2$ have a common finite cover.
\end{thm}

If the hyperplane subgroups of a compact non-positively curve cube complex are separable, then there is a finite cover such that the hyperplanes embed, do not self-osculate, and are $2$-sided.
If no inter-osculations could be added to this list, then the cube complex would be virtually special.
Conversely, specialness implies separable hyperplane subgroups, and it is conjectured that the converse holds as well.

Note that in the case $d =1$ that $L$ is the set of $n+1$ disconnected points, so the $L$-cube-complexes will be $(n+1)$-regular trees. 
If $n=2,d=2$ then $L$ is the famous Petersen graph.
In the case when $L$ has no induced squares (as in the case of the Petersen graph), the fundamental groups of $X_1$ and $X_2$ will be hyperbolic~\cite{Moussong88}, and as a consequence of Agol's proof of the virtual Haken conjecture~\cite{Agol13}, they are virtually special.
Thus we have:

\begin{cor}
 Let $n \geq 2, d = 1,2$ and let $|\Delta| = nd+1$ and $L = \frakK_n(\Delta)$.
 If $X_1$ and $X_2$ are compact $L$-cube complexes then $X_1$ and $X_2$ have common finite covers.
\end{cor}

\begin{proof}
 In the case $d =1$ the cube complexes are graphs, so it suffices to show that $L$ is square free when $d = 2$.
 Let $\Delta = \{1, \ldots, 2n +1\}$.
 Suppose that $v_1, v_2, v_3, v_4$ are the vertices of an induced square in $L$.
 Then without loss of generality we can assume that $v_1 = \{1, \ldots,n\}$ and $v_2 = \{ n+1, \ldots, 2n\}$ since they are disjoint sets.
 Thus we can further assume that $v_3 = \{2,\ldots, n, 2n+1\}$ since it must be an $n$-element set disjoint from $v_2$.
 Then we have a contradiction since $v_4$ must be an $n$-element subset disjoint from $v_1 \cup v_3 = \{1, \ldots, n, 2n+1\}$ so $v_2 = v_4$.
\end{proof}

\subsection{Strategy}
The plan is to show (in Proposition~\ref{prop:virtually_virtually_Coxeter}) that each $L$-cube-complex has a finite cover that has a finite orbi-covering $X \to X_L$, where $X_L$ is the orbi-complex $W_L \backslash D(L)$.
We seek to construct this orbi-covering by identifying the link of the $0$-cube in $X_L$ with $\frakK_n(\Delta)$ and finding a suitable map $\lk(x) \to \frakK_n(\Delta)$ for each $x$ so that the orbi-covering is defined.
By associating a copy $\Delta_x$ of $\Delta$ with each $0$-cube in $X$ we identify $\lk(x)$ with $\frakK_n(\Delta_x)$.
The orbi-covering is then locally defined by a choice of map $q_x: \Delta_x \to \Delta$ (see Lemma~\ref{lem:induced_automorphisms}).

In order for the set of $q_x$ to define an orbi-cover we need to ensure that certain conditions are satisfied.
If $e = (x,y)$ is a $1$-cube, then we need to ensure that $e$ will be mapped to the same half edge in $X_L$ by the maps induced by $q_x$ and $q_y$.
Given a square in $X$, we also need to ensure that it will be mapped to a quarter-square in $X_L$.

In Section~\ref{sec:Delta_Categories}, we formulate problem in the language of a \emph{$\Delta$-category}, which is a choice of bijection $\phi_e : \Delta_x \to \Delta_y$ for each edge $e = (x,y)$, satisfying certain conditions.
Most of the action in this paper concerns being able to (virtually) construct a $\Delta$-category.
Once we have the $\Delta$-category we obtain a holonomy 
$$\Psi: \pi_1(X,x) \to \Sym(\Delta_x)$$
and the kernel of this holonomy will give a finite cover for which we can define suitable $q_x$ (see Section~\ref{sec:Holonomy}).

\subsection{Previous results and connections to QI-rigidity}
A major motivation for proving Haglund's conjecture is the potential applications to Gromov's program of understanding groups up to quasi-isometry~\cite{Gromov87}.
In~\cite{Haglund06}, Haglund proved his conjecture for Bourdon buildings and his result can be combined with a result of Bourdon-Pajot~\cite{BourdonPajot00} which says that each quasi-isometry of such a building is finite distance from a unique automorphism.
Thus we deduce that if $G$ is a group quasi-isometric to the graph product $W$ associated to such a Bourdon building $B$, then in fact it acts by isometries on $B$.
By Agol's result~\cite{Agol13}, $G$ will be virtually special, thus acting faithfully on $B$, and by Haglund $G$ will be weakly commensurable with $W$.
Thus $W$ is quasi-isometrically rigid.

This argument motivates the following problem:

\begin{prob} \label{prob:qi_rigidity}
 Let $L = \frakK_n(\Delta)$, where $|\Delta| = nd+1$.
 Is every quasi-isometry of $D(L)$ finite distance from an automorphism?
\end{prob}

A positive answer to Question~\ref{prob:qi_rigidity} in the hyperbolic case would immediately give quasi-isometric rigidity for the associated groups $W_\Gamma$ by a similar argument to the case of Bourdon buildings.
That is to say that any group quasi-isometric to $W_\Gamma$ would be weakly commensurable with $W_\Gamma$.
In the ``higher rank'' non-hyperbolic case one might look to Huang's results on the quasi-isometric rigidity of large families of right angled Artin groups~\cite{Huang18}.
In this case following would need to be considered:

\begin{prob} \label{prob}
 Suppose that $L$ is a Kneser complex as above, such that $W_\Gamma$ is not hyperbolic.
 Are there groups acting geometrically on $D(\Gamma)$ that are not virtually special?
\end{prob}

{\bf Acknowledgements:} I would like to thank Daniel Groves and Kevin Whyte for mentioning the particularly interesting case of the Petersen graph, and Nir Lazarovich and Jingyin Huang for discussions relating to these results. I would like to thank Sam Shepherd for pointing out a mistake and suggesting the alternate separability condition on the hyperplane subgroups.

\section{Preliminaries}

\subsection{Right Angled Coxeter Groups}

We refer to Davis~\cite{Davis08} for classical background on Coxeter groups and their geometry and to~\cite{Dani20} for a recent survey of their large scale geometry.

Let $L$ denote a finite simplicial flag complex.
The right angled Coxeter group $W_L$ associated to $L$ is given by the presentation:
\[
 W_L = \langle v \in L^{(0)} \mid v^2 =1 \textrm{ and } [u,v] =1 \textrm{ if } (u,v) \in L^{(1)} \rangle. 
\]
The Davis complex $D(L)$ is the CAT(0) cube complex obtained from the Cayley graph constructed from the above presentation complex, after collapsing the $v^2$ bigons to a single edge, and inserting higher dimensional cubes whenever their $2$-skeleton appears.
The link of each vertex in $D(L)$ is isomorphic to $L$, which makes it a \emph{$L$-cube-complex}.
The following theorem tells us when $D(L)$ is the unique CAT(0) $L$-cube-complex

\begin{thm}[{\cite[Thm 1.2]{Lazarovich18}}]
 The Davis complex $D(L)$ is the unique CAT(0) cube complex with each link isomorphic to $L$ if and only if $L$ is superstar-transitive.
\end{thm}

If we colour the edges in $D(L)$ according to the corresponding element of $L$, or alternatively the conjugacy class of the associated generator, we can identify $W_L$ as the subgroup of $\Aut(D(L))$ that preserves the colours.
Sometimes this subgroup is referred to as the \emph{type-preserving automorphisms}.
The quotient $X_L = W_L \backslash D(L)$ has the structure of an orbi-complex. Each codimension-$k$-face has the associated group $(\mathbb{Z}/2)^k$ with a factor corresponding to each codimension-$1$ face that meets there.

\subsection{Kneser Complexes} \label{sec:Kneser}

Let $\Delta$ be a finite set.
The \emph{Kneser complex} $\frakK_n(\Delta)$ is the flag complex with underlying graph with vertex set given by $n$-elements subsets of $\Delta$, and edges corresponding to disjoint $n$-element subsets.
There is a natural action of $\Sym(\Delta)$ on $\frakK_n(\Delta)$.

If $\frakK := \frakK_n(\Delta)$ is a Kneser complex, then we let $\fraks_v = \fraks(v) \subseteq \Delta$ denote the subset associated to $v \in VK$.

\begin{exmp}
If $|\Delta| = 5$, then $P := \frakK_2(\Delta)$ is the \emph{Petersen graph}.
It is a simple exercise to verify that $P$ is triangle and square free.
See Figure~\ref{fig:KG(5,2)}.
\end{exmp}

\begin{figure}
 \centering
 \includegraphics[scale=0.2,keepaspectratio=true]{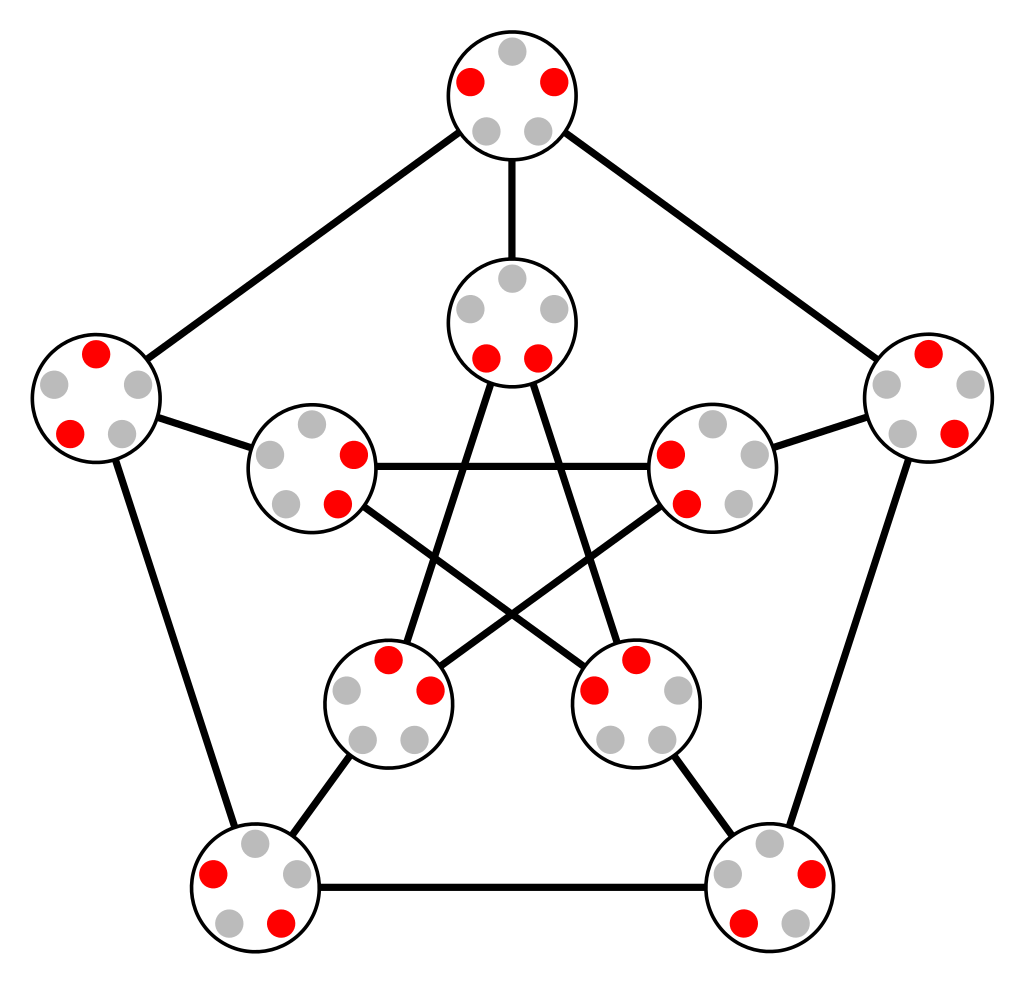}
 \caption{The Petersen graph. (Courtesy of Tilman Piesk~\cite{TilmanPiesk})}
 \label{fig:KG(5,2)}
\end{figure}

More generally, if $|\Delta| = nd +1$, then $\frakK_n(\Delta)$ is a $(d-1)$-dimensional flag simplicial complex with a superstar-transitive automorphism group (see~\cite{Lazarovich18}).
We also note the following:

\begin{lem}[{\cite[Cor 7.8.2]{GodsilRoyleBOOK}}] \label{lem:induced_automorphisms}
 If  $|\Delta| \neq 2n$, then $\Aut(\frakK_n(\Delta))$ is equal to $\Sym(\Delta)$ .
\end{lem}

Given a subset $\Sigma \subseteq \Delta$, then the inclusion induces an embeddings $\frakK_n(\Sigma) \subseteq \frakK_n(\Delta)$, where the vertex in $\frakK_n(\Sigma)$ corresponding to $\fraks \subseteq \Sigma\subseteq \Delta$ is sent to the corresponding vertex in $\frakK_n(\Delta)$.
Indeed, an automorphism $(\Delta, \Sigma) \to (\Delta, \Sigma)$ induces an automorphism of $\frakK_n(\Delta)$ that restricts to an automorphism on $\frakK_n(\Sigma)$.
Conversely, by Lemma~\ref{lem:induced_automorphisms}, provided $2n$ is not equal to $|\Delta|$ or $|\Sigma|$, an automorphism of $\frakK_n(\Delta)$ that preserves $\frakK_n(\Sigma)$ gives a self-bijection of $\Delta$ that preserves $\Sigma$.

Kneser complexes were presented by Lazarovich as a large and readily accessible set of superstar transitive graphs. 

\begin{thm}[ {\cite[Cor 5.5]{Lazarovich18}} ] \label{thm:virtually_simple}
Let $n \geq 2, d \geq 1$.
Let $|\Delta| = nd +1$ and $L := \frakK_n(\Delta)$,
then $\Aut(D(L))$ is virtually simple.
\end{thm}

We note that $D(L)$ is Gromov hyperbolic if and only if $L$ does not contain any induced squares~\cite{Moussong88}.
Thus, it is an exercise to verify that $D(L)$ is hyperbolic only if $d \leq 2$.

\begin{rem}
The most direct means that a result like Theorem~\ref{thm:main} could be true is if the automorphism group of $D(L)$ were to act properly.
In which case any other uniform lattice in $\Aut(D(L))$ would lie inside $\Aut(D(L))$ as a finite index subgroup.
Common covers of the corresponding quotient spaces could be constructed by taking the intersections of the associated lattices.
In general the automorphism groups of universal covers will be far too large for this argument to work.
Theorem~\ref{thm:virtually_simple} is the most extreme example of this: since $W_L$ is residually finite (and indeed virtually special), it cannot lie inside a virtually simple group like $\Aut(D(L))$ as a finite index subgroup.
\end{rem}

\section{Special Cube Complexes} \label{sec:Special}

We refer to~\cite{BridsonHaefliger99, Segeev14, HaglundWise08, WiseRiches, Kropholler18} for more detailed background on non-positive curvature, cube complexes, and specialness.
We outline here the terminology that we will use.

An \emph{$n$-cube} $C$ is a metric space isometrically identified with $[-1,1]^n$.
A $0$-cube is a singleton.
A \emph{subcube} $S \subseteq C$ of dimension $m$ in an $n$-cube is the $m$-cube obtained by restricting {$(m-n)$-many} coordinates to $1$ or $-1$.
The $i$-th \emph{midcube} $M \subset C$, for $1 \leq i \leq n$, is the $(n-1)$-cube obtained by restricting the $i$-th coordinate to $0$.

The \emph{reflection} of an $n$-cube over it's $i$-th midcube $M \subseteq C$ is the map $C \to C$ obtained by multiplying the $i$-th coordinate by $-1$.
Note that all the reflections in a cube commute.
The \emph{antipodal} map $C \to C$ is obtained by reflection over all the midcubes in $C$. 

The \emph{link} $\lk(x)$ of a $0$-cube $x$ in an $n$-cube $C$ is the simplex $\sigma$ given by the $\epsilon$-neighborhood of $x$ in the $\ell^1$-metric (where $1>\epsilon>0$).
Each subcube of $C$ that contains $x$ has a link at $x$ that gives a corresponding face in $\sigma$.
If $x$ and $y$ are $0$-cubes in $C$, then $x$ is mapped to $y$ by the composition $R$ of all the reflections over midcubes separating $x$ and $y$.
Thus $R$ induces an isomorphism $\lk(x) \to \lk(y)$.

By a \emph{cube complex} $X$ we will mean a topological space that decomposes into cubes $\mathcal{C}(X)$, such that every subcube of a cube in $\mathcal{C}(X)$ is a cube in $\mathcal{C}(X)$, and such that the intersection of any two cubes $C, C' \in \mathcal{C}(X)$ give subcubes of $C$ and $C'$, or the intersection is empty.
The link $\lk(x)$ of a $0$-cube $x$ in $X$ is the complex given by the union of all the links of all the cubes containing $x$, with inclusion of simplicies induced by inclusion of subcubes.
Alternatively, it can also be thought of as the $\epsilon$ neighbourhood of $x$ inside $X$ itself.
A cube complex $X$ is \emph{non-positively curved} if the link of each vertex is a simplicial flag complex.
Each an $n$-simplex $\sigma$ in $\lk(x)$ corresponds to a unique $(n +1)$-cube $C(\sigma)$ in $X$ containing $x$.
Conversely, each $(n+1)$-cube $C$ that contains $x$ corresponds to a simplex $\sigma(C)$ in $\lk(v)$.

Unless otherwise noted, our $1$-cubes will be \emph{directed} in the sense that $e = (x,y)$ comes with an initial and terminal $0$-cube, denoted $\iota e = x$ and $\tau e = y$.
The \emph{reversed $1$-cube} with the opposite direction will be denoted $\bar e = (y,x)$.
Let $X$ be a compact non-positively curved cube complex.
A hyperplane $\Lambda$ in $X$ is an equivalence class of directed $1$-cubes generated by the relation $e \sim e'$ if they are opposite faces of a square in $X$ or $ \bar e = e'$.
Associated to the equivalence class is the \emph{realization} of $\Lambda$. 
This is a non-positively curved cube complex, which we will also denote by $\Lambda$, constructed from the midcubes dual to the edges in the equivalence class that immerses by a local isometry $\Lambda \looparrowright X$.
Note that this immersion is only a cellular map when both $\Lambda$ and $X$ have been cubically subdivided.
The \emph{hyperplane subgroup} associated to $\Lambda$ is the image of $\pi_1(\Lambda)$ in $\pi_1(X)$ under the injective homomorphism given by the immersion.

A hyperplane is \emph{embedded} no two edges in the equivalence class form the corner of a square (that is to say a $2$-cube) in $X$.
Equivalently a hyperplane is embedded if the immersion of the realization is an embedding.
The \emph{carrier} $N(\Lambda) \subseteq X$ of a hyperplane is the subcomplex obtained by taking all cubes that contain an edge in the associated equivalence class.
We say that a hyperplane is \emph{fully clean and $2$-sided} if $N(\Lambda) \cong \Lambda \times [-1,1]$.
That is to say that we can extend the embedding of the realization to an embedding ${N(\Lambda) = \Lambda \times [-1,1] \hookrightarrow X}$.
If the hyperplane subgroups of $\pi_1(X)$ are separable then there is a finite cover of $X$ such that the hyperplanes are fully clean and $2$-sided (see~\cite[Section 8]{HaglundWise08}).
{\bf Thus, we will now assume going forward that all hyperplanes satisfy this condition.}
In terms of the definition of specialness, this is equivalent to the hyperplanes being $2$-sided, embedded, and without self-osculations.
Such a cube complex may fail to be special since inter-osculations do not contradict this assumption (see Figure~\ref{fig:square} for an illustration of the hyperplane pathologies.)
In terms of the assumptions of Theorem~\ref{thm:main}, if the finite index subgroups of hyperplane subgroups are separable in $\pi_1 X$, then this remains true of the hyperplane subgroups in a finite cover.

An $0$-cube $x$ is \emph{incident} to $\Lambda$ if it is contained in $N(\Lambda)$.
An edge $e = (x,y)$ is \emph{parallel} to $\Lambda$ if it is contained in $N(\Lambda)$ without being dual to $\Lambda$.
Under the assumption that the hyperplane $\Lambda$ is fully clean and $2$-sided, the immersion of the realization extends to an embedding $\Lambda \times [-1,1] \hookrightarrow X$, where the realization is the $0$ fiber.
The edges parallel to $\Lambda$ are contained in the $-1$ and $1$ fibers.
We will refer to the subcomplexes of $X$ given by the $\pm 1$ fibers as the \emph{ sides of the carrier}.

 \begin{figure}
	\begin{overpic}[width=.7\textwidth, tics=5, ]{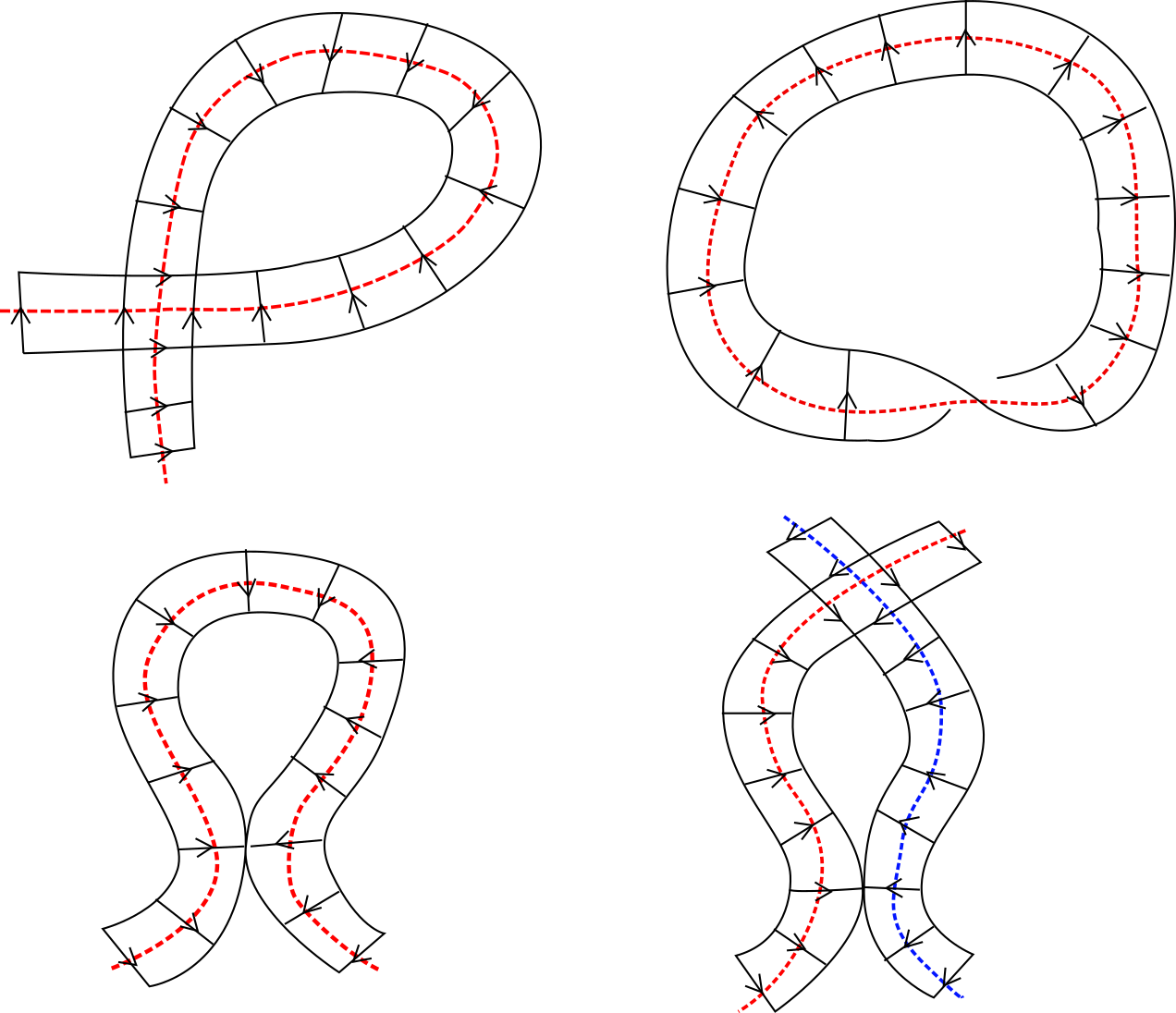} 
          \end{overpic}
	\caption{An illustration of the standard hyperplane pathologies. The dotted line depicts the topological realization of the hyperplane. The edges in the equivalence classes are given arrows indicating the direction. The top left depicts a self intersection. The top right depicts a $1$-sided hyperplane, and the edges with the arrows reversed also belong to the equivalence class. The bottom left depicts a direct self-osculation. The bottom right depicts an inter-osculation.}
	\label{fig:special}
    \end{figure}

\subsection{The adjacency map}

Let $e = (x,y)$ be an edge in $X$ dual to $\Lambda$ and let $v $ be the vertex in $\lk(x)$ corresponding to $e$, and $u $ be the vertex in $\lk(y)$ corresponding to $e$.
The \emph{star} $\cstar(\sigma)$ of a simplex $\sigma$ in a simplicial complex is the subcomplex spanned by the union of all simplicies containing $\sigma$.
We note that in~\cite{Lazarovich18} the star of a simplex is defined by Lazarovich to be the combinatorial $1$-neighbourhood. 
The two notions only coincide in the case when the simplex is the singleton.
This alternative notion, which we denote by $\cSt(\sigma)$ in the introduction, applies to the definition of superstar transitive, but will not be otherwise relevant to the content of this paper.

The \emph{adjacency map} for $e$ is the natural isomorphism
\[
 \ad_e : \cstar(v) \to \cstar(u)   
\]
such that if $v \in \sigma$ then $\ad_e(\sigma)$ is the unique simplex such that $C(\ad_e(\sigma)) = C(\sigma)$.
(This is referred to as the \emph{transfer map} in~\cite{Lazarovich18}).

More generally, let $x$ and $y$ be $0$-cubes in $X$ that belong to some $n$-cube.
Let $C$ be the minimal such $n$-cube in $X$ containing $x$ and $y$.
Let $\sigma_x \subseteq \lk(x)$  and $\sigma_y \subseteq \lk(y)$ the simplicies corresponding to $C$, the we have a natural adjacency map for $C$ given by the natural isomorphism
\[
 \ad_C : \cstar(\sigma_x) \to \cstar(\sigma_y)   
\]
such that if $\sigma$ is a simplex in $\lk(x)$ containing $\sigma_x$, then $\ad_C$ on $\sigma$ is induced by the composition of reflections in $C(\sigma)$ over the midcubes separating $x$ and $y$.
Note that $C(\ad_C(\sigma)) = C(\sigma)$.

Furthermore, suppose that $x,y,z$ are $0$-cubes in $C$ such that $C$ is the minimal cube containing $x$ and $z$.
If $C_1$ and $C_2$ are the minimal subcubes in $C$ containing $x,y$ and $y,z$ respectively, then $\ad_C = \ad_{C_1} \circ \ad_{C_2}$ where each $\ad_{C_i}$ is suitably restricted.

\section{Constructing $\Delta$-Categories.} \label{sec:Delta_Categories}

This section will be devoted to constructing a $\Delta$-category on a compact $L$-cube-complex $X$ such that all finite index subgroups of the hyperplane subgroups are separable in $\pi_1 X$, where $L$ is the Kneser graph as specified in the statement of Theorem~\ref{thm:main}.
We will assume, as stated in Section~\ref{sec:Special}, that we have passed to a finite index cover such that the hyperplanes are fully clean and $2$-sided.

\subsection{A note on notation}

In what follows we will be constructing a category over a cube complex.
We will be doing this by assigning objects to $0$-cubes and assigning morphisms to each $1$-cube. 
For example we might denote the morphism associated to $e$ by $\phi_e$.
In this case, given an edge path $\gamma = (e_1, \ldots, e_n)$, we will let $\phi_\gamma$ denote the composition $\phi_{e_n} \circ \cdots \circ \phi_{e_1}$.
If all the edges are parallel to a given hyperplane $\Lambda$, then we will call $\gamma$ a \emph{parallel path}.

\subsection{Our objects}

Let $n \geq 2, d\geq 1$ and $\Delta$ a finite set with $|\Delta| = nd +1$.
Let $X$ be a compact, non-positively curved cube complex with $2$-sided hyperplanes such that $\lk(v)$ is isomorphic to $\frakK_n(\Delta)$.
We assume, as in the statement of Theorem~\ref{thm:main}
To each zero-cube $x$ in $X$ let $\Delta_x$ be a copy of $\Delta$, and identify $\lk(v)$ with the associated Kneser complex $\frakK_n(\Delta_x)$.
%
Let $\Lambda$ be a hyperplane incident to $x$.
Let $e$ be the one-cube dual to $\Lambda$ with $\tau e = x$.
%
Let $v = \sigma(e)$ be the vertex in $\lk(v)$ corresponding to $e$.
The identification of $\lk(v)$ with $\frakK(\Delta_x)$ allows us to define ${\Lambda}_x := \fraks(v) \subseteq \Delta_x$.
We will also let $\fraks(e) := \fraks(v)$ when is is clear which $0$-cube link we are working with.
This is well defined since $\Lambda$ is fully clean, so the $1$-cube $e$ is the only $1$-cube dual to ${\Lambda}$ incident to $x$.

\subsection{$\Delta$-Categorys on $X$}

\begin{defn}
A \emph{$\Delta$-category on $X$} is a collection of bijections 
$$
  \phi_e : \Delta_{x} \to \Delta_{y},
$$
one for each $1$-cube $e = (x,y)$ in $X$, such that the following conditions are satisfied: 
\begin{enumerate}
 \item  { \bf Invertibility:}  if $e$ is a directed one cube then $\phi_{\bar e} = \phi_e^{-1}$,
 \item  Let $e_1 = (x,y), e_2 = (y,z), e_1'= (y',z), e_2'= (x, y')$ be the edges bounding a square $S$, and $ \Lambda^i$ the hyperplane dual to $e_i$ and $e_i'$ (see Figure~\ref{fig:square}). Then:
  \begin{enumerate} 
   \item  {\bf Commutativity:} $$\phi_{e_2} \circ \phi_{e_1} = \phi_{e_2'} \circ \phi_{e_1'}.$$
   \item {\bf Parallel Transport:} $$ \phi_{e_1}(\Lambda_x^2) = \Lambda_y^2 \; \textrm{ and } \; \phi_{e_2'}(\Lambda_x^1) = \Lambda_{y'}^1.$$
   \end{enumerate}
\end{enumerate}
\end{defn}

   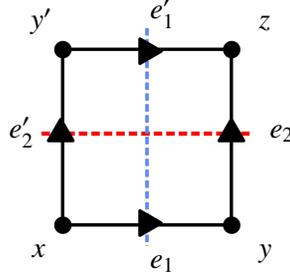
\begin{figure}
	\begin{overpic}[width=.2\textwidth, tics=5, ]{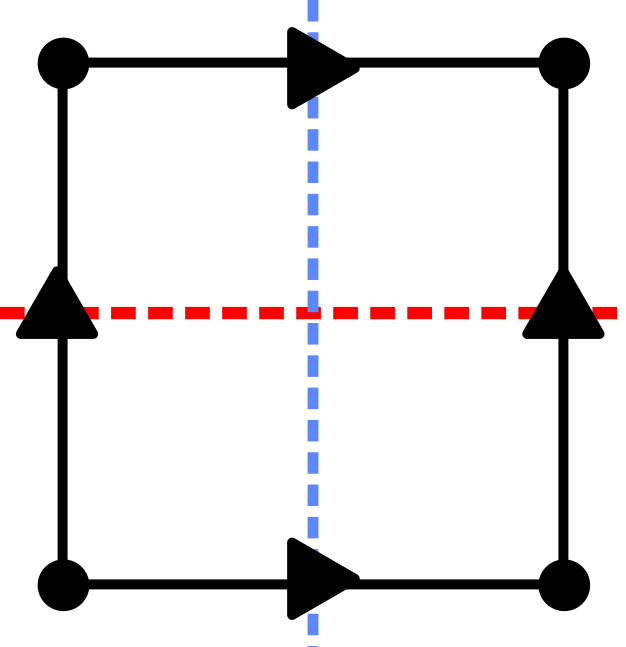} 
	    \put(-15,50){$e_2'$}
	    \put(105, 50){$e_2$}
	    \put(50,-10){$e_1$}
	    \put(50,105){$e_1'$}
	    \put(-5,-5){$x$}
	    \put(-5,100){$y'$}
	    \put(100,-5){$y$}
	    \put(100,100){$z$}
          \end{overpic}
	\caption{The square. The hyperplane $\Lambda^1$ is depicted as the vertical dotted line with the arrows on $e_1$ and $e_1'$ giving the direction. The hyperplane $\Lambda^2$ is the horizontal dotted line with the arrows on $e_2$ and $e_2'$ giving the direction.}
	\label{fig:square}
    \end{figure}

\begin{rem} \label{rem:parallel_transport}
 The parallel transport condition applied to all squares containing $e$ allows us to deduce that $\phi_{e_1}( \Lambda^1_x) = \Lambda^1_y$.
\end{rem}

Let $\{\phi_e \}$ be a $\Delta$-category on $X$, and $f: \hat X \to X$ a cover.
By identifying each link $\lk(\hat x)$ in $\hat X$ with $\frakK_n(\hat \Delta_{\hat x})$, where $\hat \Delta_{\hat x}$ is the copy of $\Delta$ assigned to $\hat x$, by Lemma~\ref{lem:induced_automorphisms} the induced isomorphism between the links
\[
 f_{\hat x} : \lk(\hat x) \to \lk(f(\hat x))
\]
induces an isomorphism 
\[
 \frakf_{\hat x} : \hat{ \Delta}_{\hat x} \to { \Delta}_{ x}.
\]
Thus we can lift the $\Delta$-category $\{\phi_e\}$ on $X$ to a unique $\Delta$-category on $\hat X$ such that the following diagram commutes:

\begin{equation*}
	\begin{tikzcd}[
	ar symbol/.style = {draw=none,"#1" description,sloped},
	isomorphic/.style = {ar symbol={\cong}},
	equals/.style = {ar symbol={=}},
	subset/.style = {ar symbol={\subset}}
	]
	\hat \Delta_{\hat x} \ar{d}{\frakf_{\hat x}} \ar{r}{ \hat \phi_{\hat e} } & \hat \Delta_{\hat y} \ar{d}{ \frakf_{\hat y} } \\
	\Delta_x \ar{r}{ \phi_{e} } & \Delta_y \\
	\end{tikzcd}
\end{equation*}
It is straight forward to verify that $\{ \hat \phi_{\hat e} \}$ satisfies the invertibility and commutativity conditions, since $\frakf_{\hat x}$ is invertible, and since the squares in $X$ lift to squares in $\hat X$.

\subsection{Constructing a $\Delta$-category}

We will construct our $\Delta$-category in two stages.
In the first stage we will define functions $\phi_e^*$ that will be defined on subsets of the domain $\Delta_x$.
We note that in this section we will be composing functions whose domain and ranges will be subsets of larger set.
In this case the composition will be given by restricting to the intersection of the corresponding domains and ranges.

\begin{lem} \label{lem:pre_system}
 There exists a unique family of functions
 \[
  \Big\{ \; \phi_e^*:  (\Delta_x -  \Lambda_x) \to (\Delta_y -  \Lambda_y) \; \Big| \;  \textrm{$\Lambda$ dual to $e= (x,y) \in X^{(1)}$} \; \Big\}
 \]
 such that

 \begin{enumerate}
 \item \label{item:inverses}  $\phi^*_{\bar e} = (\phi^*_e)^{-1}$,
 \item \label{item:commutes}  let $e_1 = (x,y), e_2 = (y,z), e_1'= (y',z), e_2'= (x, y')$ be the edges bounding a square $S$, and $ \Lambda^i$ the hyperplane dual to $e_i$ and $e_i'$ (see Figure~\ref{fig:square}). Then
 \begin{enumerate}
  \item \label{item:commutes} after suitably restricting domains 
 $$\phi_{e_2}^* \circ \phi_{e_1}^* = \phi_{e_2'}^* \circ \phi_{e_1'}^* :
 (\Delta_x -  \Lambda^1_x - \Lambda^2_x) \to (\Delta_z - \Lambda_z^1- \Lambda_z^2),
 $$
 \item \label{item:consistent}  $\phi^*_{e_1}(\Lambda_x^2) = \Lambda_y^2. \; \textrm{ and } \; \phi^*_{e_2'}(\Lambda_x^1) = \Lambda_{y'}^1.$
  \end{enumerate}
\end{enumerate}
\end{lem}

\begin{proof}
 
 Let $e = (x,y)$ be a directed $1$-cube in $X$ dual to $\Lambda$.
 Let $v$ be the vertex in $\lk(x)$ corresponding to $\bar e$, and $u$ be the vertex in $\lk(y)$ corresponding to $e$. 
  Then $\cstar(v)$ decomposes as the simplicial join $v \ast \frakK_n(\Delta_x - \Lambda_x)$ and similarly $\cstar(u)$ decomposes as $u \ast \frakK_n(\Delta_y - \Lambda_y)$.
  Thus the adjacency map $\ad_e$ restricts to an isomorphism 
  $$ \frakK_n(\Delta_x - \Lambda_x) \to \frakK_n(\Delta_y - \Lambda_y).$$
  Since $|\Delta_x - \Lambda_x| = |\Delta_y - \Lambda_y| = n(d-1) +1$, by Lemma~\ref{lem:induced_automorphisms} this isomorphism is induced by the bijection:
  \[
   \phi_e^* : (\Delta_x - \Lambda_x) \to (\Delta_y - \Lambda_y).
  \]
  (This requires checking that $n(d-1) +1 \neq 2n$ for $n \geq 2, d \geq 1$).
 
 In the case that $d = 1$ there are no squares in $X$, so conditions~\ref{item:commutes} and~\ref{item:consistent} are satisfied automatically.
 So we assume $d \geq 2$.
 Suppose that $e_1 = (x,y), e_2 = (y,z), e_1'= (y',z), e_2'= (x, y')$ are the edges bounding a square $S$, and $ \Lambda^i$ is the hyperplane dual to $e_i$ and $e_i'$ (see Figure~\ref{fig:square}).
 We now check that conditions~\ref{item:commutes} and~\ref{item:consistent} are satisfied.
 
 Verifying~\ref{item:consistent} follows from observing that $\Lambda_x^2 \subseteq \Delta_x - \Lambda_x^1$ corresponds to a vertex ${u \in \lk(x) = \frakK_n(\Delta_x)}$ and $\Lambda_y^2 \subseteq \Delta_y - \Lambda_y^1$ corresponds to a vertex $v$ in ${\lk(y) = \frakK_n(\Delta_y)}$ such that that $\ad_{e_1}(u) = v$. (Stare at Figure~\ref{fig:square}).
 Thus $\phi^*_{e_1}(\Lambda_x^2) = \Lambda_y^2$ and similarly $\phi^*_{e_2'}(\Lambda_x^1) = \Lambda_{y'}^1$.
 
 We now consider~\ref{item:commutes}. 
  Observe that~\ref{item:consistent} implies
 \begin{align*}
  \phi_{e_2}^* \circ \phi^*_{e_1}((\Delta_{x} - \Lambda_x^2) - \Lambda_x^1) & = \phi_{e_2}^* ((\Delta_{y} - \Lambda_{y}^1) - \Lambda_{y}^2) \\
               & =  \Delta_{z} - \Lambda_{z}^1 - \Lambda_{z}^2. 
 \end{align*}
 Combined with the similar set of equalities for $\phi_{e_2'}^* \circ \phi_{e_1'}^*$ this verifies ~\ref{item:consistent} when $d = 2$ since there is only one possible map between singletons.

 In the case that $d >2$, 
 let $\sigma_x \subseteq \lk(x), \sigma_y \subseteq \lk(y), \sigma_{y'} \subseteq \lk(y'), \sigma_z \subseteq \lk(z)$ denote the $1$-simplicies corresponding to the square $S$.
 We know that 
 $$
 \ad_S = \ad_{e_2} \circ \ad_{e_1} = \ad_{e_1'} \circ \ad_{e_2'} : \cstar(\sigma_x) \to \cstar(\sigma_z).
 $$
 We also have the decomposition
 $$
 \cstar(\sigma_x) = \sigma_x \ast \frakK_n(\Delta_x - \Lambda_x^1 - \Lambda_x^2)
 $$
 and similar decompositions for the stars of $\sigma_y, \sigma_{y'}$, and $\sigma_z$.
 The adjacency map $\ad_S$ therefore restricts to an isomorphism
 \[
  \frakK_n(\Delta_x -  \Lambda_x^1 - \Lambda_x^2) \to 
  \frakK_n(\Delta_z - \Lambda_z^1 - \Lambda_z^2)
 \]
 which, by Lemma~\ref{lem:induced_automorphisms}, is induced by an isomorphism
 \[
  (\Delta_x - \Lambda_x^1 - \Lambda_x^2) \to 
  (\Delta_z - \Lambda_z^1 - \Lambda_z^2)
 \]
 that must coincide with the composition $\phi_{e_2}^* \circ \phi_{e_1}^*$ and $\phi_{e_1'}^* \circ \phi_{e_2'}^*$, as their  restrictions induce the same isomorphism.
 Thus $\phi_{e_2}^* \circ \phi_{e_1}^* = \phi_{e_1'}^* \circ \phi_{e_2'}$.
\end{proof}

We will refer to the maps $\{\phi_e^*\}$ as the \emph{pre-$\Delta$-category}.
Note that if $f: \hat X \to X$ is a cover, then we can lift the pre-$\Delta$-category to $\hat X$.
Indeed, since such pre-$\Delta$-categories are unique, lifting is somehow unnecessary -- you just need to check that the corresponding diagram commutes.

\subsection{The hyperplane parallel holonomy}

As a consequence of Lemma~\ref{lem:pre_system} we deduce that if an edge $e = (x,y)$ is parallel to $ \Lambda$ then we have an bijection:
\[
 \psi_p:  \Lambda_x \to  \Lambda_y
\]
 obtained by  restricting $\phi_e^*$ as given by Lemma~\ref{lem:pre_system}.
 Indeed, if $\Lambda'$ is the hyperplane dual to $e$, then $\Lambda_x \subseteq \Delta_x - \Lambda'_x$.
 We note that this is a category, with objects that are $n$-elements sets, and morphisms that are bijections.
 
 Thus if we fix a choice of side of $\Lambda$ and a $0$-cube $p$ in $\Lambda$ as a basepoint, we obtain a \emph{parallel holonomy}
 \[
  \Psi_{p} : \pi_1(\Lambda, p) \to \Sym({\Lambda}_x).
 \]
  If $e'$ is the edge dual to $\Lambda$ with midpoint $p$ such that $\tau e' = x$ lies on the given side, this holonomy is given by identifying $\Lambda$ with the side of the hyperplane carrier containing the basepoint $x$, and letting the equivalence class of a parallel path $[\gamma] = [e_1, \ldots, e_n]$ based at $x$ map to
  \[
   \Psi_{p}([\gamma]) =  \psi_\gamma,
  \]
  where $\psi_{\gamma}$ denotes the composition $\psi_{e_n} \circ \ldots \circ \psi_{e_1}$.
  Conditions~\ref{item:inverses} and~\ref{item:commutes} in Lemma~\ref{lem:pre_system} ensure that this does not depend on the choice of representative.
  
  We note that the triviality of the holonomy does not depend on the choice of basepoint $p$ (but may depend on the side of the carrier that is chosen).
  Indeed, given another $1$-cube $e''$ dual to $\Lambda$, with $\tau e'' = y$ on the same side of $\Lambda$, with midpoint $p'$ we can check the following diagram commutes:
  \begin{equation*}
	\begin{tikzcd}[
	ar symbol/.style = {draw=none,"#1" description,sloped},
	isomorphic/.style = {ar symbol={\cong}},
	equals/.style = {ar symbol={=}},
	subset/.style = {ar symbol={\subset}}
	]
	\pi_1(\Lambda, p) \ar{d}{} \ar{r}{ \Psi_p } & \Sym(\Lambda_x) \ar{d}{ } \\
	\pi_1(\Lambda, p') \ar{r}{ \Psi_{p'} } & \Sym(\Lambda_y) \\
	\end{tikzcd}
\end{equation*}
  We have chosen some path $\gamma$ connecting $x$ to $y$ in $\tau(\Lambda)$.
  The left vertical map is given by conjugating closed loops by $[\gamma]$, in the standard fashion, and the right vertical map is given by conjugating by $\psi_{\gamma}$.

  The kernel of $\Psi_p$ is a finite index normal subgroup of $\pi_1(\Lambda)$, and by the assumptions of Theorem~\ref{thm:main} will be separable in $\pi_1 X$.
  
  \begin{lem} \label{lem:trivial_parallel}
   There exists a finite cover $\hat X \to X$ such that the parallel holonomies in $\hat X$ are trivial.
  \end{lem}

  \begin{proof}
    Let $\Psi$ be a parallel holonomy for some hyperplane $\Lambda$, and some choice of side and basepoint.
    The kernel of $\Psi$ is a finite index normal subgroup of $\pi_1(\Lambda)$, and therefore, by the assumption of Theorem~\ref{thm:main}, will be separable in $\pi_1 X$.
    Let $\{\id , g_1, \ldots, g_\ell\}$ be a minimal set of representatives for the left cosets of $\ker(\Psi)$ in $\pi_1(\Lambda)$.
    As $g_i \notin \ker(\Psi)$, by separability there exists a finite index subgroup $N_i \leqslant \pi_1(X)$ such that $\ker(\Psi) \subseteq N_i$ and $g_i \notin N_i$.
    Thus $\ker(\Psi) = \bigcap_{i=1}^\ell N_i \cap \pi_1(\Lambda)$, since we know $\ker(\Psi) \subseteq \bigcap_{i=1}^\ell N_i$ and that if $g_ih \in \bigcap_{i=1}^\ell N_i \cap \pi_1(\Lambda)$, where $h \in \ker(\Psi)$, then $g_i \in \bigcap_{i=1}^\ell N_i$.
    The normal core, $\Core(\bigcap_{i=1}^\ell N_i)$, is a finite index normal subgroup of $\pi_1 X$ such that $\pi_1(\Lambda) \cap \Core(\bigcap_{i=1}^\ell N_i)$ is contained in $\ker(\Psi)$.
    
    By repeating this for each side of each hyperplane, and intersecting all the resulting normal cores, we obtain a finite index normal subgroup $N \leqslant \pi_1(X)$ such that for each hyperplanes $\Lambda$, the intersection $N \cap \pi_1( \Lambda)$ is contained in the kernel the parallel holonomies on either side of $\Lambda$.
    The the desired finite cover $\hat X \to X$ is given by $N$.
     Let $\{ \hat \phi_{\hat e}^* \}$ denote the lift of the pre-$\Delta$-category on $X$ to $\hat X$.
 Then the following diagram commutes, telling us that the parallel holonomies in $\hat X$ are trivial.
 
\begin{equation*}
	\begin{tikzcd}[
	ar symbol/.style = {draw=none,"#1" description,sloped},
	isomorphic/.style = {ar symbol={\cong}},
	equals/.style = {ar symbol={=}},
	subset/.style = {ar symbol={\subset}}
	]
	\pi_1(\hat \Lambda) \ar{d}{\hat \Psi_{\hat e}} \ar{r}{ f_* } & \pi_1(\Lambda) \ar{d}{ \Psi_e } \\
	\Sym(\hat \Delta_{\hat x}) \ar{r}{ } & \Sym(\Delta_x) \\
	\end{tikzcd}
\end{equation*}
 The hyperplane $\hat \Lambda$ covers $\Lambda$, and the bottom arrow is the isomorphism induced by conjugation by $\frakf_{\hat x}$.
  \end{proof}

\subsection{Extending the maps $\phi_e^*$}

By Lemma~\ref{lem:trivial_parallel}, we now assume that we have passed to a suitable finite cover such that $X$ has trivial parallel holonomies in its pre-$\Delta$-category.
Given an edge $e$ dual to $\Lambda$, it remains to extend $\phi_e^*$, and this means making a choice of bijection $\Lambda_x \to \Lambda_y$.
We can certainly make such choices so that the inversion condition~\ref{item:commutes} is satisfied, and condition~\ref{item:consistent} holds as it holds for $\phi_e^*$.
It therefore remains to ensure we can make our choices so that the commutativity condition~\ref{item:commutes} is satisfied.

For each hyperplane let $e = (x,y)$ be a choice of edge edge dual to $\Lambda$.
We make a choice of map 
$$\phi_e^\circ : \Lambda_x \to \Lambda_y$$
that extends $\phi_e^*$ to $\phi_e$.

Suppose that $e'$ is some other edge dual to $\Lambda$ such that $\tau e'$ lies on the same side of $\Lambda$ as $\tau e$.
Then let $\gamma = (e_1, \ldots, e_p)$ be an edge path parallel to $\Lambda$ that connects $\tau e$ to $\tau e'$.
We also let $\gamma' = (e_1', \ldots, e_q')$ be an edge path parallel to $\Lambda$ that connects $\iota e$ to $\iota e'$.
Then we define
\[
 \phi_{e'}^\circ = \psi_{e_1}^{-1} \circ \cdots \circ \psi_{e_p}^{-1} \circ \phi_e^\circ \circ \psi_{e_q'} \circ \cdots \circ \psi_{e_1'},
\]
where $\psi_{e_i}$ and $\psi_{e_i'}$ are the parallel holonomies on either side of $\Lambda$.
Since the parallel holonomies are trivial, $\phi_{e'}^\circ$ will not depend on the choice of paths $\gamma$ and $\gamma'$.
We let $\phi_{\bar e}^\circ = (\phi_{e}^{\circ})^{-1}$ and recover that $\phi_{\bar e'}^\circ = (\phi_{e'}^\circ)^{-1}$.

It remains to check that $\{ \phi_e \}$, as defined, satisfy our commutativity relations.
Let $e_1 = (x,y), e_2 = (y,z), e_1'= (y',z), e_2'= (x, y')$ be edges bounding a square, and let $ \Lambda^i$ be the hyperplane dual to $e_i$ and $e_i'$ (see Figure~\ref{fig:square}).
Then we consider the following separate cases:
\[
 \phi_{e_2} \circ \phi_{e_1} = 
 \begin{cases}
  \phi_{e_2}^* \circ \phi_{e_1}^* : (\Delta_x - \Lambda_x^1 - \Lambda_x^2) \to (\Delta_z - \Lambda_z^1 - \Lambda_z^2) \\
  \phi^\circ_{e_2} \circ \phi_{e_1}^* : \Lambda_x^2 \to \Lambda_z^2 \\
  \phi^*_{e_2} \circ \phi_{e_1}^{\circ} : \Lambda_x^1 \to \Lambda_z^1 \\
  \phi_{e_2}^\circ \circ \phi_{e_1}^\circ : \emptyset \to \emptyset \\
 \end{cases}
\]

It follows from Lemma~\ref{lem:pre_system} that $\phi_{e_2}^* \circ \phi_{e_1}^* = \phi_{e_1'}^* \circ \phi_{e_2'}^*$.
By considering the parallel holonomies with respect to $\Lambda^2$ we can see that $$\phi_{e_2}^\circ \circ \phi_{e_1}^* = \phi_{e_2}^\circ \circ \psi_{e_1} = \psi_{e_1'} \circ \phi_{e_2'}^\circ = \phi^*_{e_1'} \circ \phi_{e_2'}^\circ.$$
A similar sequence of equalities gives that $\phi^*_{e_2} \circ \phi_{e_1}^{\circ} = \phi_{e_1'}^\circ \circ \phi_{e_2'}^*.$
Altogether this allows us to conclude that $\phi_{e_2} \circ \phi_{e_1} = \phi_{e_1'} \circ \phi_{e_2}$, and that $\{ \phi_e \}$ is a $\Delta$-category, and that we have proven the following:

\begin{prop} \label{prop:Delta_category}
 Let $n \geq 2, d\geq 1$ and let $\Delta$ be a finite set of cardinality $nd+1$.
 Let $L$ be the Kneser complex $\frakK_n(\Delta)$.
 Suppose that $X$ is an $L$-cube-complex with separable hyperplane subgroups.
 Then there exists a finite cover $\hat X \to X$, such that there is a $\Delta$-category over $X$.
\end{prop}

\section{The Holonomy} \label{sec:Holonomy}

 Given a $\Delta$-category $\{ \phi_e \}$ for $X$ we obtain a holonomy map:
 $$
  \Phi_x: \pi_1(X, x) \to \Sym(\Delta_x)
 $$
  where the homotopy class $[\gamma] = [e_1, \ldots, e_n]$ of the edge path based at $x$ has image
  $$
   \Phi_x([\gamma]) = \phi_{\gamma}.
  $$
  The invertibility and commutativity conditions guarantee that this does not depend on the choice of representative of the homotopy class.
  Note that if $\Phi_x$ is trivial, then the holonomy is trivial with respect to any basepoint since the following diagram commutes:
  
  \begin{equation*}
	\begin{tikzcd}[
	ar symbol/.style = {draw=none,"#1" description,sloped},
	isomorphic/.style = {ar symbol={\cong}},
	equals/.style = {ar symbol={=}},
	subset/.style = {ar symbol={\subset}}
	]
	\pi_1(X, x) \ar{d}{} \ar{r}{ \Phi_x } & \Sym(\Delta_x) \ar{d}{ } \\
	\pi_1(X, y) \ar{r}{ \Phi_y } & \Sym(\Delta_y) \\
	\end{tikzcd}
\end{equation*}
  If $\gamma$ is an edge path connecting $x$ to $y$, then the vertical left arrow is the isomorphism given by conjugating a homotopy class of based loops by $[\gamma]$, and the vertical right arrow is the isomorphism given by conjugating by $\phi_\gamma$.

  The kernel of $\Phi_x$ is a finite index normal subgroup of $\pi_1 X$ and corresponds to a finite sheeted, regular cover $f: \hat X \to X$.
  Lift the $\Delta$-category on $X$ to a $\Delta$-category $\{\hat \phi_{\hat{e}} \}$ on $\hat{X}$.
  We can check that the following diagram commutes:
  
  \begin{equation*}
	\begin{tikzcd}[
	ar symbol/.style = {draw=none,"#1" description,sloped},
	isomorphic/.style = {ar symbol={\cong}},
	equals/.style = {ar symbol={=}},
	subset/.style = {ar symbol={\subset}}
	]
	\pi_1(\hat X,\hat x) \ar{d}{f_*} \ar{r}{ \hat \Phi_{\hat x} } & \Sym(\Delta_{\hat x}) \ar{d}{ } \\
	\pi_1(X, x) \ar{r}{ \Phi_x } & \Sym(\Delta_x) \\
	\end{tikzcd}
\end{equation*}
  The zero-cube $\hat x$ is chosen so that $f(\hat x) = x$, and the right vertical arrow is the isomorphism given by conjugation by $\frakf_{\hat x}$.
  Thus we conclude that the holonomy $\hat \Phi_x$ on $\hat X$ is trivial.
  If the holonomy on $X$ obtained from a $\Delta$-category is trivial, then we say that the $\Delta$-category itself is \emph{flat}.

  \subsection{Constructing the orbi-cover.}

  \begin{prop} \label{prop:virtually_virtually_Coxeter}
    Let $L = \frakK_n(\Delta)$ where $|\Delta| = nd+1$.
    Let $X$ be a compact $L$-cube-complex that has a flat $\Delta$-category on $X$.
    Then there is an orbi-complex cover $X \to X_L$, where $X_L = W_L \backslash D(L)$.
  \end{prop}
  
  \begin{proof}
   Let $\{ \phi_e \}$ be the flat $\Delta$-category on $X$.
   For a basepoint $x$ fix an identification ${q_x : \Delta_x \to \Delta}$.
   For any other $0$-cube $y$ in $X$, let $q_y = q_x \circ \phi_\gamma$ where $\gamma$ is an edge path connecting $y$ to $x$.
   Note that $q_y$ does not depend on the choice of $\gamma$ since the $\Delta$-category is flat.
   
   We will prove the claim by producing a orbi-complex cover $X \to X_L$.
   First we map all $0$-cubes in $X$ to the unique $0$-cube in $X_L$.
   We can extend $X$ to the $1$-skeleton of $X$ by mapping each $1$-cube $e = (x,y)$ dual to $\Lambda$ to the half-$1$-cube corresponding to $q(\Lambda_x)$.
   This makes sense since we know that $q_x(\Lambda_x) = q_y \circ \phi_e (\Lambda_x) = q_y(\Lambda_y)$ by Remark~\ref{rem:parallel_transport}, so $e$ and $\bar e$ are mapped to the same half edge.
   
   Now we want to extend $X^{(1)} \to X_L$ to the $2$-skeleton.
   Let  ${e_1 = (x,y)}$ , ${e_2 = (y,z)}$ , ${e_1'= (y',z)}$ , ${e_2'= (x, y')}$ be the directed one cubes bounding a square $S$ in $X$ such that $e_i$ and $e_i'$ are dual to the hyperplane $\Lambda^i$ (as in Figure~\ref{fig:square}).
   We want to show that $e_i$ and $e_i'$ map to the same half edge, and the $e_1$ and $e_2'$ map to half edges that bound a quarter-square in $X_L$.
   The first fact follows from the parallel transport property since $\phi_{e_1}(\Lambda_x^2) = \Lambda_y^2$ so $q_x(\Lambda_x^2) = q_y \circ \phi_{e_1}(\Lambda_x^2) = q_y(\Lambda_y^2)$.
   The second fact follows from the fact that $\Lambda^1_x \cap \Lambda^2_x = \emptyset$ since $e_1$ and $e_2'$ bound the corner of a square, so $q_x(\Lambda^1_x) \cap q_x(\Lambda^2_x) = \emptyset$.
   
   It is immediate that we can extend $X^{(2)} \to X_L$ to the entire skeleton since the higher dimension cubes are entirely determined by the $1$-skeleton.
   Thus we can lift this orbi-covering to an isomorphism $\wt{X} \to D(L)$ such that the deck transformation group $\pi_1(X)$ is a subgroup of $W_L$.
  \end{proof}

  \begin{proof}[Proof of Theorem~\ref{thm:main}]
   Let $X_1$ and $X_2$ be our $L$-complexes.
    As the hyperplane subgroups are separable, by Proposition~\ref{prop:Delta_category} there is a finite cover $X_i' \to X_i$ such that there is a $\Delta$-category over $X_i'$.
    By considering the holonomy given by the $\Delta$-category, we can pass to a further finite cover $\hat X_i \to X_i'$ such that the induced $\Delta$-category is flat.
   By Proposition~\ref{prop:virtually_virtually_Coxeter}, there are finite orbi-covers  $f_i : \hat X_i \to X_L$.
   The common cover is then obtained by taking the intersection of the corresponding deck transformation groups inside of $W_L$.
  \end{proof}

\bibliographystyle{plain}
\bibliography{Ref.bib}

\begin{thebibliography}{10}

\bibitem{Agol13}
Ian Agol.
\newblock The virtual {H}aken conjecture.
\newblock {\em Doc. Math.}, 18:1045--1087, 2013.
\newblock With an appendix by Agol, Daniel Groves, and Jason Manning.

\bibitem{Angluin80}
Dana Angluin.
\newblock Local and global properties in networks of processors (extended
  abstract).
\newblock In {\em Proceedings of the Twelfth Annual ACM Symposium on Theory of
  Computing}, STOC '80, page 82–93, New York, NY, USA, 1980. Association for
  Computing Machinery.

\bibitem{BassKulkarni90}
Hyman Bass and Ravi Kulkarni.
\newblock Uniform tree lattices.
\newblock {\em J. Amer. Math. Soc.}, 3(4):843--902, 1990.

\bibitem{BourdonPajot00}
Marc Bourdon and Herv\'{e} Pajot.
\newblock Rigidity of quasi-isometries for some hyperbolic buildings.
\newblock {\em Comment. Math. Helv.}, 75(4):701--736, 2000.

\bibitem{BridsonHaefliger99}
Martin~R. Bridson and Andr\'{e} Haefliger.
\newblock {\em Metric spaces of non-positive curvature}, volume 319 of {\em
  Grundlehren der Mathematischen Wissenschaften [Fundamental Principles of
  Mathematical Sciences]}.
\newblock Springer-Verlag, Berlin, 1999.

\bibitem{BurgerMozes97}
Marc Burger and Shahar Mozes.
\newblock Finitely presented simple groups and products of trees.
\newblock {\em C. R. Acad. Sci. Paris S\'{e}r. I Math.}, 324(7):747--752, 1997.

\bibitem{Dani20}
Pallavi Dani.
\newblock The large-scale geometry of right-angled coxeter groups.
\newblock In {\em Handbook of group actions. {V}}, volume~48 of {\em Adv. Lect.
  Math. (ALM)}, pages 107--141. Int. Press, Somerville, MA, [2020] \copyright
  2020.

\bibitem{Davis08}
Michael~W. Davis.
\newblock {\em The geometry and topology of {C}oxeter groups}, volume~32 of
  {\em London Mathematical Society Monographs Series}.
\newblock Princeton University Press, Princeton, NJ, 2008.

\bibitem{GodsilRoyleBOOK}
Chris Godsil and Gordon Royle.
\newblock {\em Algebraic graph theory}, volume 207 of {\em Graduate Texts in
  Mathematics}.
\newblock Springer-Verlag, New York, 2001.

\bibitem{Gromov87}
M.~Gromov.
\newblock Hyperbolic groups.
\newblock In {\em Essays in group theory}, volume~8 of {\em Math. Sci. Res.
  Inst. Publ.}, pages 75--263. Springer, New York, 1987.

\bibitem{Haglund06}
Fr\'{e}d\'{e}ric Haglund.
\newblock Commensurability and separability of quasiconvex subgroups.
\newblock {\em Algebr. Geom. Topol.}, 6:949--1024, 2006.

\bibitem{HaglundWise08}
Fr\'{e}d\'{e}ric Haglund and Daniel~T. Wise.
\newblock Special cube complexes.
\newblock {\em Geom. Funct. Anal.}, 17(5):1551--1620, 2008.

\bibitem{Huang18}
Jingyin Huang.
\newblock Commensurability of groups quasi-isometric to {RAAG}s.
\newblock {\em Invent. Math.}, 213(3):1179--1247, 2018.

\bibitem{Kropholler18}
Robert Kropholler.
\newblock Special cube complexes.
\newblock In {\em Geometric and cohomological group theory}, volume 444 of {\em
  London Math. Soc. Lecture Note Ser.}, pages 46--66. Cambridge Univ. Press,
  Cambridge, 2018.

\bibitem{Lazarovich18}
Nir Lazarovich.
\newblock On regular {${\rm CAT}(0)$} cube complexes and the simplicity of
  automorphism groups of rank-one {${\rm CAT}(0)$} cube complexes.
\newblock {\em Comment. Math. Helv.}, 93(1):33--54, 2018.

\bibitem{Leighton82}
Frank~Thomson Leighton.
\newblock Finite common coverings of graphs.
\newblock {\em J. Combin. Theory Ser. B}, 33(3):231--238, 1982.

\bibitem{Moussong88}
Gabor Moussong.
\newblock {\em Hyperbolic {C}oxeter groups}.
\newblock ProQuest LLC, Ann Arbor, MI, 1988.
\newblock Thesis (Ph.D.)--The Ohio State University.

\bibitem{Neumann10}
Walter~D. Neumann.
\newblock On {L}eighton's graph covering theorem.
\newblock {\em Groups Geom. Dyn.}, 4(4):863--872, 2010.

\bibitem{Segeev14}
Michah Sageev.
\newblock {$\rm CAT(0)$} cube complexes and groups.
\newblock In {\em Geometric group theory}, volume~21 of {\em IAS/Park City
  Math. Ser.}, pages 7--54. Amer. Math. Soc., Providence, RI, 2014.

\bibitem{Shepherd19}
Sam Shepherd.
\newblock Two generalisations of {L}eighton's theorem.
\newblock With an appendix by Giles Gardam and Daniel Woodhouse.
  arXiv:1908.008302.

\bibitem{TilmanPiesk}
Piesk Tilman.
\newblock Public Domain.

\bibitem{WiseThesis}
Daniel~T. Wise.
\newblock {\em Non-positively curved squared complexes: {A}periodic tilings and
  non-residually finite groups}.
\newblock ProQuest LLC, Ann Arbor, MI, 1996.
\newblock Thesis (Ph.D.)--Princeton University.

\bibitem{WiseRiches}
Daniel~T. Wise.
\newblock {\em From riches to raags: 3-manifolds, right-angled {A}rtin groups,
  and cubical geometry}, volume 117 of {\em CBMS Regional Conference Series in
  Mathematics}.
\newblock Published for the Conference Board of the Mathematical Sciences,
  Washington, DC; by the American Mathematical Society, Providence, RI, 2012.

\bibitem{Woodhouse21}
Daniel~J. Woodhouse.
\newblock Revisiting {L}eighton's theorem with the {H}aar measure.
\newblock {\em Math. Proc. Cambridge Philos. Soc.}, 170(3):615--623, 2021.

\end{thebibliography}

\end{document}